\documentclass[12pt]{amsart}
\usepackage{amscd, amssymb,latexsym,amsmath, amscd, amsmath}
\usepackage[all]{xy}
\usepackage{anysize}
\usepackage[sc]{mathpazo} 
\usepackage{color}
\usepackage{enumitem}
\usepackage{pst-node}
\marginsize{2.5cm}{2.5cm}{2.5cm}{2.5cm}
\newtheorem{theorem}{Theorem}[section] 
\newtheorem{lemma}[theorem]{Lemma}
\newtheorem{corollary}[theorem]{Corollary}
\newtheorem{proposition}[theorem]{Proposition}
\theoremstyle{definition}

\theoremstyle{remark}
\newtheorem{remark}{Remark}

\newcommand{\U}{\mathcal{U}}

\newcommand{\wt}{\widetilde}

\newcommand{\C}{\mathbb{C}}

\newcommand{\N}{\mathbb{N}}

\newcommand{\Nm}{\mathbb{N}m}

\newcommand{\Hom}{{\rm Hom}}

\def\SL{{\rm SL}}
\def\SU{{\rm SU}}

\def\SU{{\rm SU}}
\def\U{{\rm U}}

\def\GL{{\rm GL}}
\def\PGL{{\rm PGL}}

\def\Gal{{\rm Gal}}

\begin{document}

\title[Distinguished representations for $SL(n)$]
{Distinguished representations for $SL(n)$}
\author{U. K. Anandavardhanan and Dipendra Prasad}

\address{Department of Mathematics, Indian Institute of Technology Bombay, Mumbai - 400076, India.}
\email{anand@math.iitb.ac.in}

\address{School of Mathematics, Tata Institute of Fundamental Research, Mumbai - 400005, India.}
\email{prasad.dipendra@gmail.com}


\date{}

\begin{abstract}
For $E/F$ a quadratic extension of local fields, and $\pi$ an irreducible admissible 
generic representation of $\SL_n(E)$, we calculate  the dimension of $\Hom_{\SL_n(F)}[\pi,\C]$ and relate it to fibers of the base
change map corresponding to base change of representations of $\SU_n(F)$ to $\SL_n(E)$ as suggested in \cite{pra16}. We also deal with
finite fields.
\end{abstract}

\maketitle
{\hfill \today}
\tableofcontents

\section{Introduction}\label{intro}
The paper \cite{pra16} formulates a  general conjecture - in terms of Langlands parameters, more specifically in terms of fibers of a certain base change map - on the dimension of the space 
$\Hom_{G(F)}[\pi,\C]$ 
for an irreducible admissible representation $\pi$ of $G(E)$ where $G$ is a general 
reductive group over a local field $F$, and $E/F$ is a quadratic extension of fields. 
In this paper, we consider the case of $G = \SL_n$. 
The main theorem of this paper, Theorem \ref{thm-sln}, computes $\dim_\C \Hom_{\SL_n(F)}[\pi,\C]$ 
for an irreducible admissible generic representation $\pi$ of $\SL_n(E)$ 
in terms of the fiber of the base change map from $\SU(n)$ to $\SL_n(E)$, and 
thus confirms the general conjecture in \cite{pra16} for $G=\SL_n$. 
The dimension of $\Hom_{\SL_n(F)}[\pi,\C]$ 
was computed earlier in \cite{ap03} when $n=2$ 
and in \cite{ana05} when $n$ is odd.
This paper could be considered a natural sequel to these two works, but now considered more from the point of view of base change from unitary groups. 

The symmetric space $(\SL_2(E),\SL_2(F))$ studied in \cite{ap03} was
 the first example in the literature which is not a supercuspidal Gelfand pair, that is to say the symmetric 
space affords irreducible supercuspidal representations with multiplicity $>1$.
In contrast with the $n=2$ case, when $n$ is odd, it was proved in \cite{ana05} that the symmetric space 
$(\SL_n(E),\SL_n(F))$ 
is a Gelfand pair, i.e., for any irreducible admissible representation $\pi$ of $\SL_n(E)$, $\dim_\C \Hom_{\SL_n(F)}[\pi,\C] \leq 1.$ 
In this paper we reconsider the multiplicity one theorem 
 of \cite{ana05} for  $(\SL_n(E),\SL_n(F))$ 
as well as go a little further for $n$ even.

The pair $(\SL_n(E),\SL_n(F))$ is much simpler than the general pair $(G(E),G(F))$ among other things because the 
adjoint group of $\SL_n(E)$, i.e., $\PGL_n(E)$,  operates transitively on an $L$-packet of $\SL_n(E)$, 
and in fact $\PGL_n(F)$ operates 
transitively on those representations of $\SL_n(E)$ in a given generic $L$-packet of $\SL_n(E)$ for which  
$\Hom_{G(F)}[\pi,\C] \not = 0$, and clearly $\dim_\C\Hom_{G(F)}[\pi,\C]$ is the same for all 
 representations of $\SL_n(E)$ which are conjugate under $\PGL_n(F)$. 

Before we end the introduction, let us briefly describe the main ingredients in this work. There are two non-obvious inputs in our work.
First, a recent work of Matringe describes exactly which generic representations of $\GL_n(E)$ are distinguished by $\GL_n(F)$ \cite{mat11}. This 
allows one to make some headway into understanding      $\dim_\C \Hom_{\SL_n(F)}[\pi,\C]$ where 
$\pi$ is an irreducible, admissible generic representation of $\SL_n(E)$ which  is distinguished by $\SL_n(F)$ and 
is contained in an irreducible representation $\wt{\pi}$ of $\GL_n(E)$ distinguished by $\GL_n(F)$. Second, we are able to say that inside $\wt{\pi}$ the only irreducible, admissible representation of $\SL_n(E)$ which  are 
distinguished by $\SL_n(F)$ are conjugates of $\pi$ by $\GL_n(F)$, which follows from a  more precise result according to
which an irreducible, admissible generic representation of $\SL_n(E)$ which  is distinguished by $\SL_n(F)$ must have a Whittaker model for a non-degenerate character of $N(E)/N(F)$ where $N$ is the group of upper-triangular unipotent matrices. This is a consequence of some recent
work of the first author with Matringe \cite{am16}, for which we have  given a more direct proof but one which is valid only for tempered representations, or more generally unitary representations.
 
Most of the paper is written both for $p$-adic as well as finite fields since methods are essentially uniform, and since $\dim_\C \Hom_{\SL_n(F)}[\pi,\C]$ for $F$ a finite field 
was not known in any precise way in the literature.
\section{Preliminaries}\label{prelim}

In this paper, $E/F$ is a quadratic extension of either a $p$-adic or a finite field. Let $\widetilde{G}=\GL_n(E)$, $\widetilde{H}=\GL_n(F)$, $G=\SL_n(E)$, and $H=\SL_n(F)$. An irreducible admissible representation of $\widetilde{G}$ is denoted by $\widetilde{\pi}$ and that of $G$ is denoted by $\pi$. Let $\sigma$ be the non-trivial element of the Galois group Gal$(E/F)$. Let $\N m:E^\times \rightarrow F^\times$ be the norm map.
If $F$ is $p$-adic, the quadratic character of $F^\times/N_{E/F}(E^\times)$ is denoted by $\omega_{E/F}$; if $F$ is finite, we let $\omega_{E/F}=1$. 

For a $p$-adic field $k$, let $W^\prime_k$ be its Weil-Deligne group. A Langlands parameter of $W^\prime_k$ valued in $\GL(n,\mathbb C)$, for some $n$, is typically denoted by $\widetilde{\rho}$ and a Langlands parameter of $W^\prime_k$ valued in $\PGL_n(\mathbb C)$ is typically denoted by $\rho$. 

For a representation $\tau$ of a group, $\tau^\vee$ stands for the contragredient representation, 
and $\omega_\tau$ denotes its central character (if it has one). For a representation $\tau$ of $\widetilde{G}$ or $G$, $\tau^\sigma$ is the Galois conjugate representation given by $\tau^\sigma(g)=\tau(g^\sigma)$. Similarly for a Langlands parameter $\tau$ of $W^\prime_E$, its Galois conjugate is given by $\tau^\sigma(g)=\tau(\sigma^{-1} g \sigma)$. A representation $\pi$ of $\GL_n(E)$ (or its Langlands parameter) 
is said to be conjugate self-dual if $\pi^\sigma \cong \pi^\vee$. Conjugate self-dual representations of $\GL_n(E)$ (or its Langlands parameter) come in two flavors (not mutually exclusive!): 
conjugate orthogonal and conjugate symplectic; we refer to \cite{ggp12} for the definition. This paper will deal exclusively with conjugate
orthogonal representations/parameters since they are the only ones relevant for distinction by $\GL_n(F)$.

For a character $\alpha$ of $F^\times$, 
an irreducible admissible representation $\wt{\pi}$ of $\GL_n(E)$ is said to be
$\alpha$-distinguished 
if $$\Hom_{\GL_n(F)}[\wt{\pi},\alpha]\not = 0;$$
here, as elsewhere in the paper, we identify  a character $\alpha$ of $F^\times$ to a character of $\GL_n(F)$ via the determinant map $\det : \GL_n(F) \rightarrow F^\times$. If $\alpha=1$, an $\alpha$-distinguished representation is also said to be distinguished by $\GL_n(F)$. 
 
The most basic result about distinguished representations for $(\GL_n(E),\GL_n(F))$ is the following result due to Flicker which is proved by the well-known Gelfand-Kazhdan method \cite[Propositions 11 \& 12]{fli91}.
\begin{proposition}\label{prop-flicker}
If $\widetilde{\pi}$ is an irreducible admissible representation of $\GL_n(E)$ which is $\GL_n(F)$-distinguished, then 
\[\dim_\C {\rm Hom}_{\GL_n(F)}[\widetilde{\pi},\C] = 1,\] 
and furthermore, $\widetilde{\pi}^\vee \cong \widetilde{\pi}^\sigma$ (and also $\omega_{\widetilde{\pi}}|_{_{F^\times}}=1$).
\end{proposition}

The following theorem due to Matringe \cite[Theorem 5.2]{mat11} is much more precise (which builds on the earlier works on discrete series representations \cite{kab04,akt04}).
 
\begin{proposition}\label{prop-matringe}
Let $\widetilde{\pi}$ be an irreducible admissible generic representation of $\GL_n(E)$ which is conjugate self-dual. Write
\[\widetilde{\pi} \cong \Delta_1 \times \dots \times \Delta_t \]
as the representation parabolically induced from irreducible quasi square integrable representations $\Delta_i$ of $\GL_{n_i}(E)$, with
$n=n_1+\dots+n_t$, and where the $\Delta_i$'s are not linked. Then, $\widetilde{\pi}$ is distinguished with respect to $\GL_n(F)$ if and only if, after a reordering of the indices if necessary, $\Delta_{i+1}^\sigma \cong \Delta_i^\vee$,
for $i=1,3,\dots,2r-1$, for some $r$, $\Delta_i^\sigma \cong \Delta_i^\vee$, for $2r <i \leq t$, and the discrete series representations $\Delta_i$ of $\GL_{n_i}(E)$ are distinguished by $\GL_{n_i}(F)$.
\end{proposition}

The following result was known as the Flicker-Rallis conjecture, and  is now a theorem by combining \cite[Lemma 2.2.1]{mok15} (see also \cite[Theorem 8.1]{ggp12}) and \cite[Theorem 5.2]{mat11}. 

\begin{theorem}\label{conj-fr}
An irreducible admissible generic representation $\widetilde{\pi}$ of $\GL_n(E)$ is distinguished by $\GL_n(F)$  if $n$ is odd, and $\omega_{E/F}$-distinguished if $n$ is even, if and only if its Langlands parameter is in the image of the restriction map 
\[\Phi:H^1(W^\prime_F,\widehat{\U}_n) \rightarrow H^1(W^\prime_E,\GL_n(\mathbb C)),\]
where $\widehat{\U}_n$ is the Langlands dual group of a unitary group defined by a hermitian space of dimension $n$ over $E$ (which comes equipped with an action of $W^\prime_F$).
Equivalently, an irreducible admissible generic representation of $\GL_n(E)$ is distinguished by
$\GL_n(F)$ precisely when it is a conjugate orthogonal representation.
\end{theorem}

For finite fields, we have the following result due to Gow \cite[Theorem 3.6]{gow84} (see also \cite{pra99}).

\begin{proposition}\label{Gow} For $E/F$ a quadratic extension of finite fields,  an 
irreducible representation $\widetilde{\pi}$ of $\GL_n(E)$ is distinguished by $\GL_n(F)$  if and only if $\wt{\pi}^\sigma \cong \wt{\pi}
^\vee$.
\end{proposition}

As this paper deals with representations of $\SL_n(E)$ through restriction of representations 
from $\GL_n(E)$ to $\SL_n(E)$, and similarly deals with representations of special unitary groups through restriction of representations 
from unitary groups, we will need to use twisting representations of $\GL_n(E)$, or parameters of them,
 by characters of $E^\times$, 
or in the case of unitary groups, by characters of $E^\times/F^\times$.

This motivates us to introduce {\it Strong and Weak Equivalences} among representation of $\GL_n(E)$, or parameters of them. 

Two Langlands parameters of $W^\prime_E$ with values in $\GL_n(\C)$ will be said to be weakly equivalent 
if they are twists of each other by a character of $E^\times$, and they will be  
said to be strongly equivalent if they are twists of each other by a character of $E^\times/F^\times$, i.e., 

\begin{equation}\label{eq-weak}
\widetilde{\rho}_2 \sim_w \widetilde{\rho}_1 \iff \widetilde{\rho}_2 \cong \widetilde{\rho}_1 \otimes \chi {\rm ~for~} \chi:E^\times \rightarrow \mathbb C^\times, 
\end{equation}
and
\begin{equation}\label{eq-strong}
\widetilde{\rho}_2 \sim_s \widetilde{\rho}_1 \iff \widetilde{\rho}_2 \cong \widetilde{\rho}_1 \otimes \chi {\rm ~for~} \chi:E^\times/F^\times \rightarrow \mathbb C^\times. 
\end{equation}
We denote the weak (resp. strong) equivalence class by $[\cdot]_w$ (resp. $[\cdot]_s$), and the set of strong equivalence classes 
in the weak equivalence class containing a representation $\wt{\pi}$ by $[\wt{\pi}]_w/\sim_s$.

In this paper, we will use these equivalence relations among {\it conjugate orthogonal representations}.
If $\widetilde{\rho} $ is a conjugate orthogonal representation, 
the number of strong equivalence classes of $\widetilde{\rho}$ in the weak equivalence class of $\widetilde{\rho}$ 
(among conjugate orthogonal representations) 
will be  denoted by $q(\widetilde{\rho})$.

Clearly, the same notions can be defined on the class of irreducible admissible conjugate orthogonal representations of $\GL_n(E)$, 
and as for parameters, we will denote by $q(\widetilde{\pi})$
the number of strong equivalence classes of $\widetilde{\pi}$ in the weak equivalence class of $\widetilde{\pi}$ 
(among conjugate orthogonal representations).

We remark that {\it Strong and Weak Equivalences} among representation of $\GL_n(E)$ was first introduced in \cite{ana05}.

\section{Distinction for $(\SL_n(E),\SL_n(F))$}\label{sub-sln-dist} 

The subgroup of $\GL_n(E)$ defined by 
\[\GL_n(E)^+=\{g \in \GL_n(E) \mid \det g \in F^\times E^{\times n}\} = \GL_n(F)\SL_n(E)E^\times ,\]
will play an important role in our  analysis as we  consider the restriction of an irreducible 
 representation $\widetilde{\pi}$ of $\GL_n(E)$ 
to $\SL_n(E)$ in two stages. First we restrict $\wt{\pi}$  to $\GL_n(E)^+$ 
and write it as a direct sum of irreducible representations, 
and then we look at the restriction of each of these direct summands to $\SL_n(E)$. This was 
indeed the strategy employed in \cite{ap03}. 
Note the following simple lemma:

\begin{lemma}\label{prop-mrl}
All the irreducible constituents of the restriction of a representation of $\GL_n(E)^+$ to $\SL_n(E)$ 
admit the same number of linearly independent $\SL_n(F)$-invariant functionals.
\end{lemma}

\begin{proof}
Since $\GL_n(F)\SL_n(E)E^\times = \GL_n(E)^+$, all  the irreducible 
constituents of the restriction of a representation of $\GL_n(E)^+$ to $\SL_n(E)$ 
are conjugates to one another under the inner conjugation action of $\GL_n(F)$ on $\SL_n(F)$, proving the lemma.
\end{proof}

For an irreducible, admissible representation $\wt{\pi}$ of $\GL_n(E)$, define the sets $X_{\wt{\pi}}$, $X^\prime_{\wt{\pi}}$, $Y_{\wt{\pi}}$, $Z_{\wt{\pi}}$ as follows:

\begin{enumerate}
\item $X_{\wt{\pi}} = \{\alpha \in \widehat{F^\times} \mid \widetilde{\pi} \mbox{~is $\alpha$-distinguished}\},$
\item $X^\prime_{\wt{\pi}} = \{\alpha \in \widehat{F^\times} \mid \widetilde{\pi} \mbox{~is $\alpha$-distinguished or $\alpha \cdot \omega_{E/F}$-distinguished}\},$
\item $Z_{\wt{\pi}} = \{\chi \in \widehat{E^\times} \mid \widetilde{\pi} \otimes \chi \cong \widetilde{\pi}\},$
\item $Y_{\wt{\pi}} = \{\chi \in \widehat{E^\times} \mid \widetilde{\pi} \otimes \chi \cong \widetilde{\pi}, \chi|_{_{F^\times}}=1\}$.
\end{enumerate}
Observe that $Z_{\wt{\pi}}, Y_{\wt{\pi}}$ are abelian groups, whereas $X_{\wt{\pi}}, X^\prime_{\wt{\pi}}$ are just sets, 
and that characters of $E^\times$ in $Z_{\wt{\pi}}$ 
when restricted to $F^\times$ act on the sets $X_{\wt{\pi}}, X^\prime_{\wt{\pi}} 
$ by translation, giving rise to a faithful action of $Z_{\wt{\pi}}/Y_{\wt{\pi}}$ on the sets $X_{\wt{\pi}}, X^\prime_{\wt{\pi}}$. Characters in $Z_{\wt{\pi}}$ are said to be self-twists of $\wt{\pi}$.

\begin{proposition}\label{prop-qpi} 
Let $E$ be a quadratic extension of either a finite or a $p$-adic field $F$.
Let $\pi$ be an irreducible admissible representation of $\SL_n(E)$ which is distinguished by $\SL_n(F)$. 
Then 
\begin{enumerate}
\item If $n$ is odd,
\[\dim_{\mathbb C}{\rm Hom}_{\SL_n(F)}[\pi,\C]\leq 1.\]
\item If $n$ is even, 
\[\dim_{\mathbb C}{\rm Hom}_{\SL_n(F)}[\pi,\C]\leq 
{|\{\alpha \in \widehat{F^\times} \mid \widetilde{\pi} \otimes (\alpha \circ {\mathbb N}m) \cong \widetilde{\pi}, \alpha^2=1\}|}
\leq |F^\times/F^{\times 2}|;\]
in particular, for $F$ a finite field, $\dim_{\mathbb C}{\rm Hom}_{\SL_n(F)}[\pi,\C]\leq 1$ if $F$ has characteristic 2, and $\dim_{\mathbb C}{\rm Hom}_{\SL_n(F)}[\pi,\C]\leq 2$ in odd characteristics.
\end{enumerate}
\end{proposition}

\begin{proof} 
Fix $\wt{\pi}$ to be an irreducible admissible representation of $\GL_n(E)$ distinguished by $\GL_n(F)$ with $\wt{\pi}$ 
containing $\pi$ upon restriction to $\SL_n(E)$, and 
consider the vector space 
\[V={\rm Hom}_{\SL_n(F)}[\widetilde{\pi},\C]\]
of $\SL_n(F)$-invariant linear functionals on $\widetilde{\pi}$. The group $\GL_n(F)$ operates on $V$ via
\[(g \cdot \lambda)(v)=\lambda(g^{-1} \cdot v),\]
for $v \in \widetilde{\pi}$. Observe that $\SL_n(F)$ acts trivially on $V$ by the definition of $V$, and $F^\times < \GL_n(F)$ acts trivially on $V$ by our assumption on the central character of $\widetilde{\pi}$. Since,
\[\GL_n(F)/F^\times \SL_n(F) \cong F^\times/F^{\times n},\]
a finite abelian group (if $F$ is $p$-adic, assume characteristic of  $F$ does not divide $n$), 
it follows that $V$ is a direct sum of characters of $F^\times$. If
\[V = \bigoplus_{\alpha \in \widehat{F^\times}} m_\alpha \alpha,\]
then $\widetilde{\pi}$ is $\alpha$-distinguished with respect to $\GL_n(F)$ for any $\alpha$ with $m_\alpha \neq 0$. 
Notice also that $m_\alpha \leq 1$ for each $\alpha \in \widehat{F^\times}$, since 
\[\dim_{\mathbb C}{\rm Hom}_{\GL_n(F)}[\widetilde{\pi},\alpha] \leq 1,\]
by the first part of Proposition \ref{prop-flicker}. Therefore,
\[\dim_{\mathbb C}{\rm Hom}_{\SL_n(F)}[\widetilde{\pi},\C]= \left| 
X_{\wt{\pi}} \right|.\]

Note that if $\wt{\pi}$ is $\alpha$-distinguished 
for a character $\alpha: F^\times \rightarrow {\mathbb C}^\times$, then if $\wt{\alpha}$ denotes any extension
of $\alpha$ to $E^\times$, by Proposition \ref{prop-flicker} we must have,
$$(\wt{\pi} \otimes \wt{\alpha})^\sigma \cong (\wt{\pi} \otimes \wt{\alpha})^\vee.$$
This combined with the isomorphism $\wt{\pi}^\sigma \cong \wt{\pi}^\vee$ (because $\wt{\pi}$ is $\GL_n(F)$-distinguished), 
implies that: 
$$\wt{\pi} \otimes (\wt{\alpha}^\sigma \cdot \wt{\alpha}) \cong \wt{\pi},$$
or $\alpha \circ {\mathbb N}m \in Z_{\wt{\pi}}$.

Sending a character $\alpha$ of $F^\times$ to the character $\alpha \circ {\mathbb N}m $ of $E^\times$, 
defines a homomorphism, call it ${\mathbb N}m$ from $\widehat{F^\times}$ to $\widehat{E^\times}$, whose restriction to
$X_{\wt{\pi}}$ will also be denoted by the same symbol ${\mathbb N}m$,
$${\mathbb N}m: X_{\wt{\pi}} 
\longrightarrow Z_{\wt{\pi}}/Y_{\wt{\pi}}.$$
Note that $X_{\wt{\pi}} $ 
being only a set, the map ${\mathbb N}m$ on it is only a set theoretic map, but being the restriction of a group homomorphism, the fibers of this map are contained in translates of any particular element in the fiber
by `the kernel of the map' which consists of those characters $\alpha$ of $F^\times$  for which $ \alpha \circ {\mathbb N}m \in Y_{\wt{\pi}}$, 
i.e., $\pi \otimes (\alpha \circ \Nm) \cong \pi$ 
and $ \alpha \circ {\mathbb N}m|_{F^\times} = \alpha^2 =1$. By central character considerations, we already know that 
if $\chi$ and $\chi\cdot \alpha$ both belong to $X_{\wt{\pi}} $, then 
$\alpha^n=1$. Therefore if $n$ is odd,  
the map of sets ${\mathbb N}m: X_{\wt{\pi}} \longrightarrow Z_{\wt{\pi}}/Y_{\wt{\pi}},$ is injective, and if $n$ is even, 
any fiber of this map has order at most the number of characters $\alpha$ of $F^\times$ with $\pi \otimes (\alpha \circ \Nm) \cong \pi$ and $\alpha^2=1$.

It is clear that an irreducible representation of $\GL_n(F)\SL_n(E)E^\times = \GL_n(E)^+$ when restricted to $\SL_n(E)$ has $|Z_{\wt{\pi}}/Y_{\wt{\pi}}|$ 
many irreducible
components, and since $\GL_n(F)$ acts transitively on these irreducible representations of $\SL_n(E)$, 
the number of $\SL_n(F)$-invariant linear forms on $\wt{\pi}$ contributed by that irreducible representation 
of $\GL_n(E)^+$ which contains $\pi$
equals  $|Z_{\wt{\pi}}/Y_{\wt{\pi}}| \cdot \dim_{\mathbb C}{\rm Hom}_{\SL_n(F)}[\pi,\C]$.
On the other hand, the space of $\SL_n(F)$-invariant linear forms on $\wt{\pi}$ has dimension equal to $|X_{\wt{\pi}}|$. 
Thus, we get the obvious inequality:
$$|Z_{\wt{\pi}}/Y_{\wt{\pi}}| \cdot \dim_{\mathbb C}{\rm Hom}_{\SL_n(F)}[\pi,\C] \leq |X_{\wt{\pi}}|.$$

Now, the properties of the mapping ${\mathbb N}m: X_{\wt{\pi}} \longrightarrow Z_{\wt{\pi}}/Y_{\wt{\pi}}$ 
discussed earlier proves parts $(1)$ and   $(2)$ of the proposition.  
\end{proof}

\begin{remark} 
Multiplicity one property for $n$ odd was already proved in \cite{ana05} by a similar method as above.
It is not clear to the authors if this multiplicity one property is a consequence of `Gelfand's trick'.
\end{remark}
 
The following proposition refines the earlier proposition when $X_{\wt{\pi}}$ is known to be a group, for example, when $F$ is a finite field, or when $F$ is a $p$-adic field, and $\wt{\pi}$ is a discrete series representation.

\begin{proposition}\label{prop-qpii} 
Let $E/F$ be a quadratic extension of either finite or $p$-adic fields. 
Let $\pi$ be an irreducible admissible discrete series representation of $\SL_n(E)$ if $F$ is $p$-adic, and any 
irreducible representation if $F$ is finite field. Assume $\pi$  is distinguished by $\SL_n(F)$ 
and is contained in an irreducible representation $\wt{\pi}$ of $\GL_n(E)$ distinguished by $\GL_n(F)$. 
Let ${\mathbb N}m: X^\prime_{\wt{\pi}} \longrightarrow Z_{\wt{\pi}}/Y_{\wt{\pi}}$ be the norm map defined earlier. 
Let $c(F)=2$ if $F$ is a $p$-adic field, and 
$c(F)=1$ if $F$ is a finite field.
Then for $n$ an even integer, 
\begin{eqnarray*} 
c(F)\dim_{\mathbb C}{\rm Hom}_{\SL_n(F)}[\pi,\C] & = & \frac{|X^\prime_{\tilde{\pi}}|}{|Z_{\tilde{\pi}}/Y_{\tilde{\pi}}|}  
\\ &=& 
 \frac{| Ker ~~{\mathbb N}m|}{|Coker ~~{\mathbb N}m|} \\ 
& = &  \frac{|\{\chi \in \widehat{F^\times} \mid \widetilde{\pi} \otimes (\chi\circ {\mathbb N}m) \cong \widetilde{\pi}, \chi^2=1\}|}
{|\{\lambda|_{F^\times}{\rm ~~such ~~that~~}  \lambda \mbox{{\rm ~is a self-twist of~}} \tilde{\pi} \} / 2 X^\prime_{\tilde{\pi}} |}. \end{eqnarray*}
\end{proposition}
\begin{proof}The proof of this proposition follows  the same strategy which was used in the proof of the 
previous proposition by using the following additional inputs:

\begin{enumerate}
\item A discrete series representation $\wt{\pi}$ 
of $\GL_n(E)$ is $\omega_{E/F}$-distinguished or distinguished if and only if 
$$\wt{\pi}^\sigma \cong \wt{\pi}^\vee;$$
furthermore, such a representation $\wt{\pi}$ 
of $\GL_n(E)$ is either distinguished or $\omega_{E/F}$-distinguished, with exactly one possibility. This follows from \cite[Theorem 7]{kab04} and \cite[Corollary 1.6]{akt04} if $F$ 
is a $p$-adic field, and a consequence of Proposition \ref{Gow} if $F$ is finite. 
This implies in particular that $X^\prime_{\wt{\pi}}$ is a group 
 and the map  
$${\mathbb N}m: X^\prime_{\wt{\pi}} \longrightarrow Z_{\wt{\pi}}/Y_{\wt{\pi}},$$
is now a group homomorphism, whose kernel is  
\[\{\chi \in \widehat{F^\times} \mid \widetilde{\pi} \otimes (\chi \circ \Nm) \cong \widetilde{\pi}, \chi^2=1\}.\]

\item The restriction of $\wt{\pi}$ to $\GL_n(E)^+$ has exactly one irreducible representation - the one which carries 
Whittaker functional for a character of $N(E)/N(F)$ - which is distinguished 
by $\SL_n(F)$; this is the content of the next section.
\end{enumerate}
First two equalities in the statement of the proposition follows from these. For the last equality in the statement
of the proposition, observe that
\begin{enumerate}[label=(\alph*)]
\item The natural map $j:  Z_{\wt{\pi}}/Y_{\wt{\pi}} \rightarrow X_{\wt{\pi}}$ is injective, and
\item the composition of the maps: $
 X^\prime_{\wt{\pi}}\stackrel{{\mathbb N}m} \longrightarrow Z_{\wt{\pi}}/Y_{\wt{\pi}} \stackrel{j}\rightarrow X^\prime_{\wt{\pi}}$
is multiplication by $2$.
\end{enumerate}
\end{proof}

\begin{corollary}
If $n$ is even,  $F$ a finite field, and the representation $\wt{\pi}$ of $\GL_n(E)$ is distinguished by $\GL_n(F)$, 
then if $\wt{\pi}$ does not have a self-twist by the unique character of $E^\times$ 
order 2 (such a character of $E^\times$ comes from $F^\times$ through the norm map), then \[\dim_{\mathbb C}{\rm Hom}_{\SL_n(F)}[\pi,\C]\leq 1.\] 
If $\wt{\pi}$ has a self-twist by the unique character of $E^\times$ of order 2, and also by a character $\chi$ with $\chi(-1)=-1$, then also \[\dim_{\mathbb C}{\rm Hom}_{\SL_n(F)}[\pi,\C]\leq 1.\] 
\end{corollary}
\begin{proof}
Observe that the image of the map ${\mathbb N}m: X_{\wt{\pi}} \longrightarrow Z_{\wt{\pi}}/Y_{\wt{\pi}},$ 
consists of those characters on 
$F^\times$ whose value on $-1 \in F^\times$ is 1. Therefore, if there is a self-twist of $\pi$ by a character $\chi$ of $E^\times$ with $\chi(-1) \not = 1$, then the map  ${\mathbb N}m: X_{\wt{\pi}} \longrightarrow Z_{\wt{\pi}}/Y_{\wt{\pi}},$ could not be surjective. This allows one to prove the corollary.
\end{proof}

\begin{remark} 
Proposition \ref{prop-qpii}
allows us to calculate $\dim_{\mathbb C}{\rm Hom}_{\SL_n(F)}[\pi,\C]$ 
(which we already know is $\leq 2$) in all cases for $F$ a 
finite field even if in the above corollary, we have not handled all cases. We just want to add the observation  that
$\dim_{\mathbb C}{\rm Hom}_{\SL_n(F)}[\pi,\C]$ for $\pi$ an irreducible representation of $\SL_n(E)$ 
as well as $\dim_{\mathbb C}{\rm Hom}_{\GL_n(F)}[\wt{\pi},\C]$  for $\wt{\pi}$ an irreducible representation of $\GL_n(E)$, depends only on the
{\it semisimple part of the Jordan decomposition } of $\pi, \wt{\pi}$ (in the sense of Lusztig). 
\end{remark}

The next proposition follows from the method of proof of Proposition \ref{prop-qpi} (using that a generic distinguished representation
of $\SL_n(E)$ is generic for a character of $N(E)/N(F)$ for which we refer to the next section). 
For $n=2$, this proposition is \cite[Theorem 1.4]{ap03} and 
for a tempered representation $\pi$ for any $n$, this is \cite[Theorem 4.3]{ana05}).

\begin{proposition}\label{prop-qpj}
Let $\pi$ be an irreducible admissible generic representation of $\SL_n(E)$ which is distinguished by $\SL_n(F)$
and contained in an irreducible representation $\wt{\pi}$ of $\GL_n(E)$ distinguished by $\GL_n(F)$. Then,
\[\dim_{\mathbb C}{\rm Hom}_{\SL_n(F)}[\pi,\C]=\frac{|X_{\wt{\pi}}|}{|Z_{\wt{\pi}}|/|Y_{\wt{\pi}}|}.\]
\end{proposition}

\begin{remark}\label{rmk-qpi}
That the right hand side of the identity in Proposition \ref{prop-qpj} is indeed a positive integer can be observed independently. Indeed, the group $Z_{\wt{\pi}}/Y_{\wt{\pi}}$ acts freely on $X_{\widetilde{\pi}}$, and hence it is the number of orbits under this action.
\end{remark}

Our next result 
relates distinction for the symmetric space $(\SL_n(E),\SL_n(F))$ to the notion of strong and weak equivalences defined at the end of \S 2 on preliminaries.

\begin{proposition}\label{thm-qpi}
Let $\pi$ be an irreducible admissible generic representation of $\SL_n(E)$ which is distinguished by $\SL_n(F)$. Let $\widetilde{\pi}$ be an irreducible admissible generic representation of $\GL_n(E)$ which contains $\pi$ on restriction to $\SL_n(E)$, 
and is distinguished by $\GL_n(F)$. Then,
\[\dim_{\mathbb C}{\rm Hom}_{\SL_n(F)}[\pi,\C]=q(\widetilde{\pi}),\]
where $q(\widetilde{\pi})$ is the number of strong equivalence classes 
in the weak equivalence class of $\widetilde{\pi}$, i.e., the cardinality of the set $[\widetilde{\pi}]_w/\sim_s$ (inside conjugate orthogonal representations of $\GL_n(E)$).
\end{proposition}

\begin{proof}
If $\alpha$ is a character of $F^\times$ in $ X_{\widetilde{\pi}}$ 
and if $\wt{\alpha}$ 
is any extension of $\alpha$ to $E^\times$, then by the definition of $ X_{\widetilde{\pi}}$, 
$\wt{\pi}[\wt{\alpha}] 
= \widetilde{\pi} \otimes \widetilde{\alpha}^{-1}$ 
is distinguished by $\GL_n(F)$, hence by Theorem \ref{conj-fr}, 
it is a conjugate orthogonal representation, 
therefore  $\wt{\pi}[\wt{\alpha}] \in [\widetilde{\pi}]_w$; 
different extensions $\wt{\alpha}$ of $\alpha$ give rise 
to elements in a given strong equivalence class, thus 
$\widetilde{\pi} [\wt{\alpha}]$ 
as an element of $[\widetilde{\pi}]_w/\sim_s$ 
depends only on $\alpha$. Since by Theorem \ref{conj-fr}, conjugate orthogonal generic representations are precisely the 
 irreducible admissible generic representations of $\GL_n(E)$ that are distinguished by $\GL_n(F)$, the mapping
$\alpha \rightarrow \widetilde{\pi} [\wt{\alpha}]$ is surjective onto  $[\widetilde{\pi}]_w/\sim_s$.

Under the natural action of $ Z_{\wt{\pi}} /Y_{\wt{\pi}}$ on $X_{\wt{\pi}}$, it is clear that 
$\widetilde{\pi} [\wt{ \beta}\wt{\alpha} ] = \widetilde{\pi} [\wt{\alpha }] $ 
as an element of 
 $[\widetilde{\pi}]_w/\sim_s$ for $\widetilde{\beta} \in  Z_{\wt{\pi}}$.

Conversely, if $\widetilde{\pi} [\wt{\alpha}] = \widetilde{\pi} [\wt{\beta }] $ as an element of $[\widetilde{\pi}]_w/\sim_s$, then  
$\widetilde{\pi} \otimes \widetilde{\beta}^{-1} \cong \widetilde{\pi} \otimes \widetilde{\alpha}^{-1}\chi $ for some 
character $\chi$ of $E^\times/F^\times$. This condition is equivalent to saying that 
$\widetilde{\beta}\widetilde{\alpha}^{-1}\chi \in Z_{\widetilde{\pi}}.$ Therefore $\widetilde{\alpha}$ and $\widetilde{\beta}$ differ by an element of $Z_{\wt{\pi}}$.
\end{proof}

\section{Distinction by $\SL_n(F)$ and Whittaker models}\label{sec-whittaker}

In [AP03] using explicit realization of a $\GL_2(F)$-invariant linear form
in the Kirillov model of a representation $\pi$ of $\GL_2(E)$  due to Jeff Hakim, it was proved 
that any irreducible admissible generic  representation of $\SL_2(E)$ which is distinguished by
$\SL_2(F)$ has a Whittaker model for a character $\psi: E/F \rightarrow \C^\times$. This result was among the most important non-trivial ingredient to our work in [AP03]. Its analogue for $\SL_n(E)$ will be similarly crucial to us in this paper.

In a recent work of the first author with Matringe \cite{am16}, it has been proved that the integral representation 
for the invariant linear form 
\[\ell(W)=\int_{N_n(F)\backslash P_n(F)}W(p)dp\]
can be defined on the Whittaker space $\mathcal W(\widetilde{\pi},\psi)$ (absolutely convergent integral for $\wt{\pi}$ unitary \cite[Lemma 4]{fli88}, and defined by regularization in general \cite[\S 7]{am16}),  associated to an irreducible generic representation 
$\wt{\pi}$ of $\GL_n(E)$, 
and up to multiplication by scalars, is the unique non-zero element in Hom$_{\GL_n(F)}(\widetilde{\pi},1)$, which allows one to conclude
as in [AP03] that any irreducible generic representation of $\SL_n(E)$ which is distinguished by
$\SL_n(F)$ has a Whittaker model for a non-degenerate character $\psi: N(E)/N(F) \rightarrow \C^\times$.

In this section,  we offer a `pure thought' argument based on Clifford theory with the `mirabolic' subgroup of $\GL_n(E)$, 
the subgroup of $\GL_n(E)$ with last row $(0,\cdots, 0, 1)$, first for $\SL_2(E)$, and then for $\SL_n(E)$
in general but only for tempered representations. Our proof for $\SL_2(E)$ works for finite fields, 
but the proof for  $\SL_n(E)$, when $E$ is finite,  works only for cuspidal representations.

\begin{lemma} \label{AP3}
Let $\pi$ be an irreducible generic representation of $\SL_2(E)$. Then if $\pi$ is distinguished by
$\SL_2(F)$,  $\pi$ must have a Whittaker model for a character $\psi: E/F \rightarrow \C^\times$.
\end{lemma}

\begin{proof} Since $\pi$ is distinguished by $\SL_2(F)$, the largest quotient of $\pi$ on which $\SL_2(F)$ operates trivially
is non-zero. As a consequence, the largest quotient $\pi_{F}$ of $\pi$ on which $N(F) = F$ operates trivially is non-zero. 
Clearly $\pi_{F}$ is a smooth module for $N(E)/N(F) = E/F$. Thus there are two options:

\begin{enumerate}
\item  $N(E)/N(F)$ does not operate trivially on $\pi_F$, in which case it is easy to prove that
for some non-trivial character $\psi: N(E)/N(F) \rightarrow \C^\times$, $\pi_{\psi} \not = 0$.

\item   $N(E)/N(F)$ operates trivially on $\pi_F$, in which case in particular
$N(E)$ will operate trivially on the linear form $\ell: \pi \rightarrow \C$
which is $\SL_2(F)$-invariant. Thus this linear form will be invariant under
$\SL_2(F)$ as well as $N(E)$, and therefore the group generated by $\SL_2(F)$ and $N(E)$. 
It is easy to see that the group generated by $\SL_2(F)$ and $N(E)$ is $\SL_2(E)$. 
Thus $\ell: \pi \rightarrow \C$ is invariant under $\SL_2(E)$, so $\pi$ must be one dimensional, a contradiction to its being generic.
\end{enumerate}

This completes the proof of the lemma. \end{proof}

\begin{proposition} \label{AP}
Let $\pi$ be an irreducible admissible tempered representation of $\SL_n(E)$. Then if $\pi$ is distinguished by
$\SL_n(F)$, it must have a Whittaker model for a non-degenerate character $\psi: N(E)/N(F) \rightarrow \C^\times$.
\end{proposition}

We will prove this proposition in the following equivalent form.

\begin{proposition} \label{AP2}
Let $\pi$ be an irreducible admissible tempered representation of the group $\GL^+_n(E)= E^\times \GL_n(F)\SL_n(E)$. Then if $\pi$ is distinguished by
$\GL_n(F)$, then $\pi$ must have a Whittaker model for a non-degenerate character $\psi: N(E)/N(F) \rightarrow \C^\times$.
\end{proposition}

The proof of this proposition will depend on the following lemma which allows an inductive procedure to prove the previous 
proposition.

In what follows, for any $k \geq 0$, we let $\nu$ be the character $\nu(g)= |\det g|$ on $\GL_k(F)$, 
and all of its subgroups.

\begin{lemma} 
Let $_k\Delta$ be a smooth representation of $P^+_k(E)$, the mirabolic subgroup 
of $\GL_k^+(E)$, thus with $P^+_k(E)=\GL^+_{k-1}(E)\rtimes N_{k}(E)= \GL^+_{k-1}(E)\rtimes E^{k-1}$. 
Assume that $_k\Delta$ has a Whittaker model.
Fix a non-trivial character $\psi_0:E/F \rightarrow \C^\times$, and let  
$\psi_{k-1} = \psi_0\circ p_{k-1}: E^{k-1}\rightarrow \C^\times$ be the character on $E^{k-1}$ where $p_{k-1}:
E^{k-1} \rightarrow E$ is the projection to the last co-ordinate.   
 Then if $_k\Delta$ is distinguished by
$P_k(F)$, but the (un-normalized) Jacquet module $_k\Delta_{N(E)}$, a representation of $\GL^+_{k-1}(E)$ is not distinguished by 
$\GL_{k-1}(F)$,
the  smooth representation (un-normalized twisted Jacquet module) $\Delta_{N_k(E),\psi_{k-1}}$ of $P^+_{k-1}(E)$,
 must have a Whittaker model and is $\nu^{-1/2}$-distinguished 
by $P_{k-1}(F)$.
\end{lemma}

\begin{proof} 
Since $_k\Delta$ is distinguished by $P_k(F)$, 
the largest quotient 
$_k\Delta_{N_{k}(F)}$ of $_k\Delta$ on which $N_{k}(F) = F^{k-1}$ 
operates trivially is non-zero, and is distinguished by $\GL_{k-1}(F)$. 
Clearly 
$_k\Delta_{N_{k}(F)}$ is a smooth representation for $ \GL_{k-1}(F) \rtimes (N_{k}(E)/N_{k}(F)) 
= \GL_{k-1}(F) \rtimes F^{k-1}$.   
Thus we are in the context of Clifford theory which applies to any smooth representation
of a group in the presence of an abelian normal subgroup, cf. \cite[\S 5.1 C]{bz76}, for a similar analysis, and \cite[\S 3]{dp16} for developing the Clifford theory in greater generality.

Note that for $k\geq 2$, the action of $\GL_{k-1}(F)$  on the set of 
non-trivial characters of $ N_{k}(E)/N_{k}(F) = F^{k-1}$ is transitive.    

It follows from \cite[Proposition 1]{dp16} that the representation 
$_k\Delta_{N_{k}(F)}$ of  $ \GL_{k-1}(F) \rtimes (N_{k}(E)/N_{k}(F)) = \GL_{k-1}(F) \rtimes F^{k-1}$ 
has a filtration with two subquotients, which are (with un-normalized induction):

\begin{enumerate}
\item  ${\rm ind}_{P_{k-1}(F)\rtimes F^{k-1}}^{\GL_{k-1}(F) \rtimes F^{k-1}} (_k\Delta_{N_{k}(E),\psi_{k-1}}),$

\item   $_k\Delta_{N_{k}(E)}.$
\end{enumerate}

Since we know that $_k\Delta_{N_{k}(F)}$ is distinguished by $\GL_{k-1}(F)$, at least one of the representations above
is distinguished by $\GL_{k-1}(F)$. In case $(1)$, by Mackey theory, $_k\Delta_{N_{k}(E),\psi_{k-1}},$ a smooth representation of $P^+_{k-1}(E)$,
is $\nu^{-1/2}$-distinguished by 
$P_{k-1}(F)$, whereas in case $(2)$, $_k\Delta_{N_{k}(E)},$ is distinguished by $\GL_{k-1}(F)$. By hypothesis in the lemma,
$_k\Delta_{N_{k}(E)},$ is not distinguished by $\GL_{k-1}(F)$, leaving us with only option $(1)$.

This completes the proof of the lemma. 
\end{proof} 

\begin{proof}[Proof of Proposition \ref{AP2}] 
For the proof of the proposition,  we will apply the previous lemma to the representation $_k\Delta = \pi^{(n-k)}|_{P^+_{k}(E)}$ 
where $\pi^{(n-k)}$ is the $(n-k)$-th derivative of Bernstein-Zelevinsky, which is a representation of $\GL_{k}(E)$, starting with $k=n$, and 
$_n\Delta = \pi|_{P^+_n(E)}$. It follows from Bernstein-Zelevinsky that $_k\Delta_{N_{k}(E)} = \nu^{1/2} \pi^{(n-k+1)}$,  
a smooth representation of $\GL_{k-1}^+(E)$.  Further, $\nu^{1/2}{}  _k\Delta_{N_{k}(E),\psi_{k-1}} = {} _{k-1}\Delta,$ a smooth representation of $P^+_{k-1}(E)$, which implies that the way we have defined $_{k}\Delta$, decreasing induction hypothesis holds
if we can ensure that the condition, ``$_k\Delta_{N_k(E)}= \nu^{1/2} \pi^{(n-k+1)}$, a smooth representation of  $\GL^+_{k-1}(E)$,   is not distinguished by $\GL_{k-1}(F)$'', is satisfied.  This is where we will use the temperedness hypothesis. 

Recall that a tempered representation $\pi$ of $\GL_n(E)$ is of the form $\pi = \pi_1 \times \cdots \times \pi_r$ 
where $\pi_i$ are irreducible
unitary discrete series representations of $\GL_{n_i}(E)$. It is known that any unitary discrete series representation 
$\pi_i$ is the unique irreducible
quotient representation of $\rho_i \nu^{-(n_i-1)/2} \times \cdots \times  \rho_i \nu^{(n_i-1)/2}$ 
for a unitary supercuspidal representation $\rho_i$ 
of some $\GL_{m_i}(E)$ for $m_i|n_i$, and that $\pi_i^{(k)} = 0$  if $k$ is not a multiple of $m_i$, and for $k=m_ir$,
$\pi_i^{(m_ir)} $ is the unique irreducible quotient of $\rho_i \nu^{-(n_i-1)/2 +r} \times \cdots \times  \rho_i \nu^{(n_i-1)/2}$. 

The Leibnitz rule for derivatives allows one to calculate the derivative of $\pi = \pi_1 \times \cdots \times \pi_r$, and from the
recipe of the derivatives of a discrete series recalled above, we find that any non-zero positive derivative $\pi_i^{(k)}$ has a 
 central character $\omega(\pi_i^{(k)})$ 
whose absolute value   $|\omega(\pi_i^{(k)})|$ is a positive power of $\nu$ unless $k=0$ or $k=n_i$. 
Since a distinguished representation $\Lambda$ of $\GL_n(E)$ 
must have $\Lambda^\sigma \cong \Lambda^\vee$, in particular, $|\omega(\Lambda)| =1$. This implies that $\nu^{1/2}\pi^{(k)}$ cannot be 
$\GL_{n-k}(F)$ distinguished, unless it is a representation of $\GL_0(E) = 1$. 
\end{proof}

\begin{remark}
We believe that Propositions \ref {AP} and  \ref{AP2} 
remain valid for finite fields, but have not been able to find a proof, except as 
mentioned earlier in the case of cuspidal representations where the proof given here for $p$-adic fields remains valid, and 
the case of $\SL_2(E)$ independently proved in Lemma \ref{AP3}.
\end{remark} 

\begin{remark}
The proof of Proposition \ref{AP2} given here 
is based on an idea contained in \cite{akt04} that although the restriction to mirabolic of a representation of $\GL_n(E)$ has 
two subquotients, the non-generic component  cannot carry $P_n(F)$-invariant linear forms, because of the presence of the modulus character. Since the modulus character for finite fields is trivial, we are not able to rule this possibility out for finite fields. 
Note that \cite{akt04}
uses a lemma, \cite[Lemma 2.4]{akt04}, 
according to which (using un-normalized induction unlike \cite[Lemma 2.4]{akt04}),
$$ \Hom_{P_n(F)}[{\rm ind}_{P_k(E)}^{P_n(E)} (\pi \times \psi_{n-k}), \C] \cong \Hom_{\GL_k(F)}[\pi, \C],$$
where $P_n(E)$ is the mirabolic subgroup of $\GL_n(E)$, and $P_k(E)$ is the subgroup of $\GL_n(E)$ contained in the $(k,n-k)$-parabolic and containing its unipotent radical with Levi replaced by $\GL_k(E) \times {\rm U}_{n-k}(E)$ where ${\rm U}_{n-k}(E)$ is the upper triangular unipotent subgroup of $\GL_{n-k}(E)$, and $\psi_{n-k}$ is its generic character. For the proof of this lemma, \cite{akt04} refers to the main lemma of Flicker's paper \cite{fli93}, whose proof is rather long winded. Our proof here does not need \cite[Lemma 2.4]{akt04}, but rather gives a proof of it. 
\end{remark} 

\section{Fibers of the base change map from $\SU(n)$ to $\SL_n(E)$}\label{sub-parameters}

In this section we consider Langlands parameters for the groups $\SU(n)$ and $\SL_n(E)$. Our aim here is to compute the number of parameters of $\SU(n)$ that lift to a given parameter of $\SL_n(E)$. 

A Langlands parameter of $\SL_n(E)$ 
\[\phi:W^\prime_E \rightarrow \PGL_n(\mathbb C)\]
gives rise to an element of $H^1(W^\prime_E,\PGL_n(\mathbb C))$, where the Weil-Deligne group $W^\prime_E$ of $E$ acts trivially on $\PGL_n(\mathbb C)$. It is well-known that such a parameter $\phi$ lifts to a Langlands parameter $\widetilde{\phi}$ of $\GL_n(E)$
\[\widetilde{\phi}:W^\prime_E \rightarrow \GL_n(\mathbb C),\]
which can be thought of as an element of $H^1(W^\prime_E,\GL_n(\mathbb C))$ with the $W^\prime_E$-action on $\GL_n(\mathbb C)$ being trivial. Indeed, the above observation follows from Tate's theorem according to 
 which $H^2(W^\prime_E,\mathbb C^\times)= 0$ for the trivial 
action of  $W^\prime_E$ on
$\mathbb C^\times$. We  note that though Tate's theorem is usually stated in terms of the absolute Galois group $\Gal(\bar{E}/E)
$ instead of the Weil-Deligne group $W^\prime_E$, i.e., $H^2(\Gal(\bar{E}/E),\mathbb C^\times)= 0$ with $\Gal(\bar{E}/E)$ 
acting trivially on $\mathbb C^\times$ (cf. \cite[Theorem 4]{ser77}); the $W^\prime_E$-version, $H^2(W'_E,\C^\times)= 0$, and its 
relation to lifting of continuous projective representations is known too, cf. \cite[Theorem 1, Theorem 8]{raj04}.  We will
continue to call the vanishing of $H^2(W'_E,\mathbb C^\times)$ as Tate's theorem.

That a Langlands parameter for $\SL_n(E)$ lifts to a Langlands parameter for $\GL_n(E)$ 
is related to the fact that an irreducible admissible representation $\pi$ of $\SL_n(E)$ 
occurs in the restriction of an irreducible admissible representation $\widetilde{\pi}$ of $\GL_n(E)$.

As in the case of $(\GL(n),\SL(n))$, an irreducible representation of $\SU(n)$  occurs in the restriction of an irreducible admissible representation 
 of $\U(n)$.  We will check below that a Langlands parameter for $\SU(n)$ lifts to a Langlands parameter for $\U(n)$.     

Since the Langlands dual group of $\U(n)$ is
\[{^L}\U(n) = \GL_n(\mathbb C) \rtimes W^\prime_F,\]
where $W^\prime_F$ acts by projection to Gal$(E/F)$, and via
\[\sigma(g) = J {^t}g^{-1}J^{-1},\]
where $J$ is the anti-diagonal matrix with alternating $1, -1$. We will denote the group $\GL_n(\C)$ with this action of $W^\prime_F$ by $\GL_n(\C)[\tau]$; similarly for $\PGL_n(\C)$.
Thus a Langlands parameter for $\U(n)$ gives rise to 
an element of $H^1(W^\prime_F,\GL_n(\mathbb C)[\tau])$, where $W^\prime_F$ acts on $\GL_n(\mathbb C)$ as above. 
Similarly, a Langlands parameter for $\SU(n)$ gives rise to  an element of $H^1(W^\prime_F,\PGL_n(\mathbb C)[\tau])$.

Thus the fact that a Langlands parameter for $\SU(n)$ lifts to a Langlands parameter for $\U(n)$ follows from the following lemma.

\begin{lemma}\label{lem-tate}
Let $W^\prime_F$ operate on $\mathbb C^\times$ by $z \mapsto z^{-1}$ via the quotient 
$W^\prime_F \rightarrow W^\prime_F/W^\prime_E \cong \mathbb Z/2$. Denote the corresponding representation of $W^\prime_F$ by $\C^\times [\tau]$. 
Then,
\[H^2(W^\prime_F,\mathbb C^\times[\tau])=0.\]
\end{lemma}

\begin{proof}
Consider the restriction-corestriction sequence
\[H^2(W^\prime_F,\mathbb C^\times[\tau]) \rightarrow H^2(W^\prime_E,\mathbb C^\times) \rightarrow H^2(W^\prime_F,\mathbb C^\times[\tau]).\] 
Since the composite map is multiplication by $2$, and since $H^2(W^\prime_E,\mathbb C^\times)=0$ by Tate's theorem, it follows that
\[2 H^2(W^\prime_F,\mathbb C^\times[\tau]) = 0.\]
Using the exact sequence,
\[
\xymatrix@C=2pc@R=2pc{
1 \ar[r] & \mathbb Z/2 \ar[r] & \mathbb C^\times[\tau] \ar[r]^{[2]} & \mathbb C^\times[\tau] \ar[r]^{} & 1,
}
\] since $2 H^2(W^\prime_F,\mathbb C^\times[\tau]) =0$,
it follows that we have an exact sequence
\[H^1(W^\prime_F,\mathbb C^\times[\tau]) \rightarrow H^1(W^\prime_F,\mathbb C^\times[\tau]) \rightarrow H^2(W^\prime_F,\mathbb Z/2) \rightarrow H^2(W^\prime_F,\mathbb C^\times[\tau]) \rightarrow 0. \]
Now,
\[H^2(W^\prime_F,\mathbb Z/2) = \mathbb Z/2,\]
since this is the $2$-torsion in the Brauer group. Therefore, to prove that 
\[H^2(W^\prime_F,\mathbb C^\times[\tau])=0,\]
it suffices to prove that 
\[2 H^1(W^\prime_F,\mathbb C^\times[\tau]) \neq H^1(W^\prime_F,\mathbb C^\times[\tau]).\]

A cocycle in $H^1(W^\prime_F,\mathbb C^\times[\tau]) $ 
upon restriction to $W^\prime_E$ gives rise to a character of $E^\times$ which is 
trivial on elements of $F^\times$ which arise as norms from $E^\times$. It can be seen that a character  $\chi: 
E^\times/\Nm E^\times \rightarrow \C^\times$ extends to a cocycle on $W^\prime_F$ with values in $\mathbb C^\times[\tau] $ if and only if $\chi$  is trivial on $F^\times$, and then the cocycle is unique up to coboundary. Thus, 
\[H^1(W^\prime_F,\mathbb C^\times[\tau]) = {\rm Hom}(E^\times/F^\times, \mathbb C^\times[\tau]) = {\rm Hom}(U(1),\mathbb C^\times[\tau]),\]
where the second equality is the result of the identification $\chi \rightarrow \chi'$  via $\chi'(x/x^\sigma)=\chi(x)$.
Clearly, a character $\chi'$ of $U(1)$ has a square root if and only if $\chi'(-1)=1$, and therefore 
\[H^1(W^\prime_F,\mathbb C^\times[\tau])/2 H^1(W^\prime_F,\mathbb C^\times[\tau]) = \mathbb Z/2,\]
proving the lemma.
\end{proof}

We are interested in computing the number of Langlands parameters of $\SU(n)$ that lift to a given Langlands parameter of $\SL_n(E)$. Thus, we need to analyse the fiber of the  restriction map
\[
\xymatrix@C=2pc@R=2pc{
H^1(W^\prime_F,\PGL_n(\mathbb C)[\tau] 
) \ar[r]^{P\Phi} & H^1(W^\prime_E,\PGL_n(\mathbb C)).
}
\]

For this, observe that the above map fits into the following commutative diagram:
\[ 
\xymatrix@C=2pc@R=2pc{
H^1(W^\prime_F,\PGL_n(\mathbb C)[\tau] ) \ar[r]^{P\Phi} & H^1(W^\prime_E,\PGL_n(\mathbb C)) \\
H^1(W^\prime_F,GL_n(\mathbb C)[\tau] ) \ar[u]_{P_F} \ar[r]^{\Phi} & H^1(W^\prime_E,\GL_n(\mathbb C)) \ar[u]_{P_E}
}
\]
where $\Phi$ is the restriction map which corresponds to lifting a Langlands parameter of $\U(n)$ to a Langlands parameter of $\GL_n(E)$, and the maps $P_F$ and $P_E$ are 
the natural projection maps. Note that we have proved in the preceding paragraphs that both  the maps $P_F$ and $P_E$ 
 are surjective; surjectivity of $P_E$ follows from Tate's theorem and surjectivity of $P_F$ is a consequence of Lemma \ref{lem-tate}. 

The map $\Phi$ which takes a $\U(n)$-parameter to a $\GL_n(E)$-parameter is well understood, and  its image consists precisely of conjugate self-dual Langlands parameters of $W^\prime_E$ of parity $+1$ if $n$ is odd, and parity $-1$ if $n$ is even. 
We will need to make use of another well-known fact about the map $\Phi$
for which we refer to \cite[Proposition 7]{pra16} for a proof.

\begin{lemma}\label{prop-injective}
The restriction map
\[ 
\xymatrix@C=2pc@R=2pc{
H^1(W^\prime_F,\GL_n(\mathbb C)[\tau] ) \ar[r]^{\Phi} & H^1(W^\prime_E,\GL_n(\mathbb C)) 
}
\]
is injective.
\end{lemma}

We will have many occasions to use the following lemma, cf. \cite[Proposition 42]{ser02}.

\begin{lemma} \label{center} 
Suppose $G$ is a group with an action of $W^\prime_F$, and  $Z$ is a central subgroup of  $G$ left invariant by the action of $W^\prime_F$. Then elements $\phi_1,\phi_2$ of $H^1(W^\prime_F, G)$ which lie over the same element of $H^1(W^\prime_F,G/Z)$ are 
translates of each other by an element of $H^1(W^\prime_F,Z)$, i.e., $\phi_2 = \phi_1 \cdot c$ for some $c \in H^1(W^\prime_F,Z)$.
\end{lemma}  

The following proposition is a simple consequence of the previous two lemmas using the definition of strong and weak 
equivalence introduced at the end of \S2.
 
\begin{proposition}\label{thm-fiber}
Let $\rho \in H^1(W^\prime_F,\PGL_n(\mathbb C)[\tau] )$. Let $\widetilde{\rho} \in H^1(W^\prime_F,\GL_n(\mathbb C)[\tau] )$ be such that $P_F(\widetilde{\rho})=\rho$. Then the cardinality of the set
\[
\{\mu \in H^1(W^\prime_F,\PGL_n(\mathbb C)[\tau] ) \mid P\Phi(\mu)=P\Phi(\rho)\} 
\]
equals $q(\Phi(\widetilde{\rho}))$, which is the number of strong equivalence classes in the weak equivalence class of $\Phi(\widetilde{\rho})$ in the class among conjugate self-dual representations of a given parity.
\end{proposition}

\begin{proof} By Lemma \ref{center}, parameters for $\SL_n(E)$ can be identified to parameters for $\GL_n(E)$ up to 
twisting by characters $\chi: E^\times \rightarrow \C^\times$. Similarly, by Lemma \ref{center}, parameters for $\SU_n(F)$ can be identified to parameters for $\U_n(F)$ 
up to 
twisting by characters $\chi: E^\times/F^\times \rightarrow \C^\times$ (because $H^1(W^\prime_F, \C^\times[\tau]) \cong \Hom(U(1), \C^\times) = \Hom(E^\times/F^\times, \C^\times)$). By 
Lemma \ref{prop-injective}, parameters for $\U_n(F)$ embed into parameters for $\GL_n(E)$ by the base change map $\Phi$.  
Thus the cardinality of the  
fiber of the base change map  
\[\xymatrix@C=2pc@R=2pc{ H^1(W^\prime_F,\PGL_n(\mathbb C)[\tau] 
) \ar[r]^{P\Phi} & H^1(W^\prime_E,\PGL_n(\mathbb C))}\]
is the number of strong equivalence classes in the weak equivalence class of $\Phi(\widetilde{\rho})$ among conjugate self-dual representations of a given parity ($=(-1)^{n-1}$).\end{proof}

We next restate Theorem \ref{conj-fr} taking into account Lemma \ref{prop-injective} 
according to which  parameters for $\U_n(F)$ embed into parameters for $\GL_n(E)$ by the base change map $\Phi$.  

\begin{theorem}\label{thm-gln}
An irreducible admissible generic representation $\widetilde{\pi}$ of $\GL_n(E)$ is distinguished with respect to $\GL_n(F)$ 
(if $n$ is odd, and $\omega_{E/F}$-distinguished if $n$ is even) if and only if its Langlands parameter $\widetilde{\rho}_{\widetilde{\pi}}$ is in the image of
\[\Phi:H^1(W^\prime_F,\GL_n(\mathbb C)[\tau] ) \rightarrow H^1(W^\prime_E,\GL_n(\mathbb C)),\]
and moreover,
\begin{equation}\label{eq-gln}
\dim_\C {\rm Hom}_{\GL_n(F)}[\widetilde{\pi},\C] = |\Phi^{-1}(\widetilde{\rho}_{\widetilde{\pi}})|. 
\end{equation}
\end{theorem}

The main theorem of this paper is the $\SL(n)$-analogue of Theorem \ref{thm-gln}.

\begin{theorem}\label{thm-sln}
An irreducible admissible generic representation $\pi$ 
of $\SL_n(E)$ is distinguished by $\SL_n(F)$ if and only if 
\begin{enumerate}
\item its Langlands parameter $\rho_\pi$ is in the image of the base change map:
\[P\Phi:H^1(W^\prime_F,\PGL_n(\mathbb C)[\tau] ) \rightarrow H^1(W^\prime_E,\PGL_n(\mathbb C)),\]
\item $\pi$ has a Whittaker model for a non-degenerate character of $N(E)/N(F)$.
\end{enumerate}
Further, if ${\rm Hom}_{\SL_n(F)}[\pi,\C] \not = 0$,
\begin{equation}\label{eq-sln}
 \dim_\C {\rm Hom}_{\SL_n(F)}[\pi,\C] = |P\Phi^{-1}(\rho_\pi)|.
\end{equation}
\end{theorem}

\begin{proof}
Choose $\widetilde{\pi}$ as in Proposition \ref{thm-qpi} and $\widetilde{\rho}$ as in Proposition \ref{thm-fiber} so that $\Phi(\widetilde{\rho})=\widetilde{\rho}_{\widetilde{\pi}}$. Such a choice does exist by the first part of Theorem \ref{thm-gln}. Thus, the assertion $(1)$ about Langlands parameter $\rho_\pi$ follows from the commutativity of the diagram:
\[ 
\xymatrix@C=2pc@R=2pc{
H^1(W^\prime_F,\PGL_n(\mathbb C)[\tau] ) \ar[r]^{P\Phi} & H^1(W^\prime_E,\PGL_n(\mathbb C)) \\
H^1(W^\prime_F,\GL_n(\mathbb C)[\tau] ) \ar[u]_{P_F} \ar[r]^{\Phi} & H^1(W^\prime_E,\GL_n(\mathbb C)) \ar[u]_{P_E}.
}
\]
The assertion $(2)$ about Whittaker models is part of the conclusion of \S \ref{sec-whittaker}.

For the assertion on $ \dim_\C {\rm Hom}_{\SL_n(F)}[\pi,\C]$, observe that the left hand side of (\ref{eq-sln}) is $q(\widetilde{\pi})$ by Proposition \ref{thm-qpi}, whereas the right hand side of (\ref{eq-sln}) is $q(\widetilde{\rho}_{\widetilde{\pi}})$ by Proposition \ref{thm-fiber}. Since $q(\widetilde{\pi})=q(\widetilde{\rho}_{\widetilde{\pi}})$,
this proves the theorem. \end{proof}

\end{document}